\documentclass{amsart}

\usepackage{amssymb, latexsym,pdfsync,amsmath,amsthm, ulem,hyperref,graphicx}
\usepackage{fullpage}
\usepackage{pgf,tikz}
\usepackage{mathrsfs}
\usetikzlibrary{arrows}
\newtheorem{theorem}{Theorem}

\newtheorem{lemma}{Lemma}

\theoremstyle{definition}
\newtheorem{definition}{Definition}
\newtheorem{remark}{Remark}

\newtheorem*{lemma*}{Lemma}

\newcommand{\FF}{\mathbb{F}}

\newcommand{\Fq}{\mathbb{F}_q}
\newcommand{\Fqn}{\mathbb{F}_{q^n}}

\newcommand{\G}{\mathcal G}
\newcommand{\C}{\mathcal C}

\newcommand{\cS}{\mathcal S}

\def\F{\mathbb{F}}

\def\Fq{{\mathbb{F}}_q}

\def\GL{\mathrm{GL}}

\def\dim{\mathrm{dim}}

\def\Tr{\mathrm{Tr}}

\def\diag{\mathrm{diag}}

\newcommand{\npmatrix}[1]{\left( \begin{matrix} #1 \end{matrix} \right)}

\newcommand{\rank}{\mathrm{rank}}

\begin{document}
\title{The tensor rank of semifields of order 16 and 81}
\author{Michel Lavrauw and John Sheekey}\thanks{The first author acknowledges the support of {\it The Scientific and Technological Research Council of Turkey}, T\"UB\.{I}TAK (project no. 118F159).}

\maketitle

\begin{abstract}
We determine the tensor rank of all semifields of order 16 over $\F_2$ and of all semifields of order 81 over $\F_3$. Our results imply that some semifields of order 81 have lower multiplicative complexity than the finite field $\F_{81}$ over $\F_3$. We prove new results on the correspondence between linear codes and tensor rank, including a generalisation of a theorem of Brockett and Dobkin to arbitrary tensors, which makes the problem computationally feasible. 
\end{abstract}


\section{Introduction}

Tensors are fundamental mathematical objects which arise in a wide variety of pure and applied problems. The {\it tensor rank} is of great importance for many of these problems. Except for the two-fold tensors, which correspond to matrices, and in which case the theory of rank and equivalence (under the natural action of associated general linear groups) is straightforward and well understood, the testing of equivalence and the calculation of the tensor rank are both theoretically and computationally difficult for $k$-fold tensors for $k\geq 3$. See for example \cite{BuClSh1997}, \cite{GroQ}, and \cite{Halstad1990}, where it was proved that tensor rank is NP-complete for any finite field and NP-hard for the rational numbers. Recent computational results for tensors over small fields have been carried out in for example \cite{BrSt2012},  \cite{Stavrou}. 

The multiplicative structure of a (not necessarily associative) $n$-dimensional algebra over a field $\F$ can be represented by a tensor in $\F^n\otimes \F^n\otimes \F^n$, and tensor equivalence corresponds to algebra isotopism. The tensor rank (over some field containing $\F$ and contained in the {\it centre} of the algebra) then gives a measure of the complexity of multiplication in the algebra, and is invariant under equivalence, see e.g. Burgisser et al. \cite{BuClSh1997}. 

In this paper we consider the problem of computing the tensor rank of {\it finite semifields}, that is, division algebras over a finite field where multiplication is not necessarily associative. Finite semifields have been studied since Dickson constructed the first nontrivial examples in the early 1900's \cite{Dickson1905}. Semifields over the real numbers date back even further, to Graves' construction of the octonions. Tensors (and the related concept of hypercubes) have been used to study semifields by Knuth \cite{Knuth1965}, Liebler \cite{Liebler}, and the first author of this paper \cite{Lavrauw2013}, in which the tensor rank of a semifield was introduced as an isotopism invariant. 
Until now, due to the hardness of the problem, there had been no known theoretical or computational proof for the existence of semifields of the same order with different tensor rank. For example, it was shown in \cite{LaPaZa2013} that all semifields of order $3^3$ have the same tensor rank. 

In this paper we determine the tensor rank of all semifields of order $3^4$, and obtain the first proof for the existence of semifields of the same order with different tensor rank. This establishes the tensor rank as a non-trivial invariant for the isotopism classes of finite semifields. Most notably, the tensor rank of most of the semifields of order $3^4$ is one less than the tensor rank of the finite field of order $3^4$. By virtue of the relationship between the complexity of multiplication and the tensor rank, this gives these semifields a significant computational advantage for applications. The following theorem is the main result of this paper.

\begin{theorem}
The field and generalised twisted field of order $3^4$ each have tensor rank nine over $\F_3$. All other semifields of order $3^4$ have tensor rank eight over $\F_3$.
\end{theorem}

The proof consists of both a theoretical and a computational component. 

The paper is organised as follows. In Section \ref{subsec:tensors_and _contractions} we introduce some of the background on tensors, and their contractions spaces. In Section \ref{subsec:tensors_algebras} we introduce the necessary theory on semifields and their relation with nonsingular tensors. We include a first result on the automorphism group of the spread set of the finite field, which will be used in proving the main result. In the next section, Section \ref{subsec:tensors_and_codes} we use the link between tensor rank and the minimum distance of a linear code. This section includes a slight generalisation of the result by Brockitt and Dobkin from \cite{BrDo1978}, which will assist us in speeding up the algorithms and make them computationally feasible. In the final subsection of Section \ref{sec:prelim} we give an overview of the relevant known results on tensor ranks.

Based on these theoretical results, in Section \ref{sec:algorithms} we desribe the algorithms which were used to determine the tensor rank of all semifields of order $2^4$ and $3^4$. 

In Section \ref{sec:16} it is shown how the theoretical results allow for the computation of the tensor rank of all semifields of order $2^4$. This leads to Theorem \ref{thm:16} which confirms that every semifield of order $2^4$ has tensor rank $9$ over $\F_2$. 

Finally, in Section \ref{sec:81}, the theoretical results are applied to the computation of the tensor rank over $\F_3$ of semifields of order $3^4$, leading to the main result of this paper. Computationally this is by far the most significant part of our results. We provide the necessary computational data which should allow the reader to verify our results.



\section{Definitions and Preliminary Results}\label{sec:prelim}

\subsection{Tensors and Contraction Spaces}\label{subsec:tensors_and _contractions}


Here we recall the necessary definitions regarding tensors. 

Let $V=\bigotimes_{i=1}^t V_i$ denote the tensor product of the finite-dimensional vector spaces $V_i$, $i\in \{1,\ldots,t\}$. The elements $v=v_1\otimes \cdots \otimes v_t$ are called {\it fundamental tensors}, or {\it pure tensors}. We say that elements of $V$ are tensors of {\it order $t$}.

Every tensor $T\in V$ can be written as a sum of pure tensors. We refer to an expression 
$$T = \sum_{j=1}^R v_{j1}\otimes \cdots \otimes v_{jt}$$
as a {\it decomposition} of $T$ into the sum of $R$ pure tensors.

\begin{definition}
The {\it tensor rank} of $T$ is the minimum nonnegative integer $R$ such that there exists a decomposition of $T$ into $R$ pure tensors. It is denoted by $\mathrm{trk}(T)$.
\end{definition}

The pure tensors are then by definition the tensors of rank one.

The group $G= \prod_i \GL(V_i)$ acts on pure tensors in the natural way: 
$$(g_1,\ldots,g_t) ~:~v_1\otimes \cdots \otimes v_t \mapsto g_1(v_1)\otimes \cdots \otimes g_t(v_t).$$
This can be extended to define an invertible linear transformation on $V$, and so we have that $G\leq \GL(V)$. The group $G$ setwise stabilises the set of pure tensors in $\GL(V)$. 

Two tensors are {\it $G$-equivalent} (or just {\it equivalent} if $G$ is clear from the context) if they are in the same $G$-orbit of $V$. Clearly, $G$-equivalent tensors have the same tensor rank.

If $v_i^\vee$ is an element of the dual space $V_i^\vee$ of $V_i$, we define 
the contraction of $v=v_1\otimes\ldots \otimes v_t$ by $v_i^\vee$ as
\[
v_i^\vee(v) := v_i^\vee(v_i) (v_1\otimes\cdots\otimes  v_{i-1}\otimes v_{i+1}\otimes\cdots\otimes v_t)\in \bigotimes_{j\ne i}V_j.
\]
This can be extended linearly to define $v_i^\vee(T)$ for any tensor $T\in V$. Given a subspace $U$ of $V$ we can also naturally define $v_i^\vee(U)$, which is a subspace of $\bigotimes_{j\ne i}V_j$. An important concept is that of the so-called {\it contraction space} of a tensor, defined as follows.


\begin{definition}
The {\it $i$-th contraction space} of a tensor $T$ is the space
\[
\C_i(T) := \{v_i^\vee(T):v_i^\vee \in V_i^\vee\},
\]
which is a subspace of $\bigotimes_{j\ne i}V_j$ of dimension at most $\dim(V_i)$. If $\dim(\C_i(T))=\dim(V_i)$ then we say that $T$ is {\it $i$-concise}. If $T$ is $i$-concise for each $i\in\{1,\ldots,t\}$ then $T$ is called {\it concise}.
\end{definition}

The following is well known, see e.g. \cite{Lavrauw2013} for a proof.

\begin{lemma}\label{lem:span}
The tensor rank of $T\in V$ is equal to the minimum number of pure tensors in $\bigotimes_{j\ne i}V_j$ whose span contains $\C_i(T)$.
\end{lemma}

The group $G_i= \prod_{j\ne i} \GL(V_j)$ acts on $\bigotimes_{j\ne i}V_j$. We define the {\it automorphism group} of $\C_i(T)$ to be the setwise stabiliser of $\C_i(T)$ in $G_i$.

We will mainly focus on the case $V_1\simeq V_2\simeq V_3\simeq \Fq^n$, and on the first contraction space $\C_1(T)$. Denoting by $M_n(\F_q)$ the $\Fq$-vector space of $n\times n$ matrices with entries from $\Fq$, we may identify $v_1^\vee(T)$ with an element of $M_n(\Fq)$, and $\C_1(T)$ with a subspace of $M_{n}(\Fq)$ of dimension at most $n$.

\subsection{Tensors, Algebras, and Semifields}\label{subsec:tensors_algebras}

Let $V=\Fq^n$ be an $n$-dimensional vector space over $\Fq$. Let us recall the following correspondence between 3-fold tensors and algebras. We follow the notation and terminology from \cite[Section 4.2]{Lavrauw2013} to which we refer for further details on this correspondence.

Given a tensor $T = \sum_i a_i^\vee\circ b_i^\vee\circ z_i$ in $V^\vee\otimes V^\vee\otimes V$, it is well known that we can define an algebra $A_T=(V,\circ_T)$ with (not necessarily associative) $\Fq$-bilinear multiplication $\circ_T:V\times V\rightarrow V:(x,y)\mapsto x\circ_T y$ by
\[
x\circ_T y := \sum_i a_i^\vee(x)b_i^\vee(y)z_i.
\]
Conversely, an algebra $A=(V,\circ)$ with $\Fq$-bilinear multiplication $\circ$ defines a homomorphism from $V\otimes V$ to $V$, mapping $x\otimes y$ to $x\circ y$. As $\mathrm{Hom}(V\otimes V,V)$ is isomorphic to $V^\vee\otimes V^\vee\otimes V$, from $\circ$ we can define a tensor $T_A$. With appropriate choices of isomorphisms, we have that $T_{A_T}=T$, and $A_{T_A}=A$. 

As $(\Fq^n)^\vee$ is isomorphic to $\Fq^n$ as a vector space, we may fix an isomorphism and abuse notation a little by identifying $T_A$ with an element of $\Fq^n\otimes \Fq^n\otimes \Fq^n$.

\begin{definition}
The {\it tensor rank} of an algebra $A$ over $\Fq$ is defined as the tensor rank of the tensor $T_A$.
\end{definition}

Note that an algebra with centre $\Fq$ can also be regarded as an algebra over any subfield of $\Fq$. For example the finite field $\FF_{p^2}$ is a one-dimensional algebra over itself, in which case it has tensor rank one, but also as a two-dimensional algebra over $\FF_p$, in which case it has tensor rank three (see \cite{LaSh2014}). For the rest of this paper we will fix a field $\Fq$ and consider all algebras as $n$-dimensional algebras over $\Fq$. We will refer to its tensor rank as its tensor rank over $\Fq$, even if its centre is in fact strictly larger than $\Fq$.

\begin{definition}
Two algebras $A=(V,\circ)$ and $B=(V,\star)$ with centre containing $\Fq$ are said to be {\it isotopic} if there exist invertible $\Fq$-linear maps $f,g,h\in \GL(V)$ such that
\[
f(x\circ y)=g(x)\star h(y)
\]
for all $x,y\in V$. The triple $(f,g,h)$ is called an {\it isotopism}.
\end{definition}

We state the following lemma without proof; see for example \cite{Lavrauw2013}, where the result is proved for a special case, but the proof is identical.

\begin{lemma}
Two algebras $A$ and $B$ over $\Fq$ are isotopic if and only if the tensors $T_A$ and $T_B$ are equivalent. In particular, if two algebras over $\Fq$ are isotopic, then their tensor ranks over $\Fq$ are equal.
\end{lemma}

In this paper we are particularly interested in algebras known as {\it semifields}.

\begin{definition}
A {\it finite presemifield over $\Fq$} is a vector space $\Fq^n$ together with an $\Fq$-bilinear multiplication $\circ:\Fq^n\times \Fq^n\rightarrow\Fq^n:(x,y)\mapsto x\circ y$ such that $x\circ y=0$ if and only if $x=0$ or $y=0$. If this multiplication possesses an identity element, then we refer to it as a {\it finite semifield}.
\end{definition}
In other words, a finite semifield is a finite division algebra in which multiplication is not necessarily associative. We note that the definition for infinite semifields is more nuanced; we will limit ourselves to the finite case here. If we consider the extension field $\Fqn$ as a vector space over $\Fq$, and represent multiplication in $\Fqn$ by juxtaposition, then we have the following important examples.
\begin{itemize}
\item 
The field $\Fqn$ with multiplication $x\circ y = xy$.
\item
The {\it Generalized Twisted Fields (GTFs)} of Albert \cite{Albert1961} with multiplication $x\circ y = xy-cx^{q^i}y^{q^j}$, where $c$ is a fixed element of $\Fqn$ with $c^{\frac{q^n-1}{q-1}}\ne 1$, and $i,j\in \{1,\ldots,n-1\}$, $i\ne j$.
\end{itemize}

There are many other known constructions; see \cite{LaPo2011} for an overview. 
From a presemifield $\cS$ we can obtain a tensor $T_\cS$ as outlined above. The tensors corresponding to presemifields are precisely the {\it nonsingular tensors} \cite{Lavrauw2013}.
All presemiifelds are isotopic to a semifield, using Kaplansky's trick; see \cite{LaPo2011}. For theoretical classifications of semifields we refer to \cite{LaPo2011}. Full computational classifications of finite semifields have been obtained only for $q=2, n\leq 6$; $q=3, n\leq 5$; and $q\leq 5, n\leq 4$; see \cite{Demp2008}, \cite{Rua2009}, \cite{Rua2011a}. Complete computational classifications are also known for certain special type of semifields, see for example \cite{LaRo2018}, in which all 8-dimensional rank 2 commutative semifields are classified, and \cite{LaSh202?} in which four-dimensional commutative semifields over $\F_7$ have been classified.

\begin{definition}
A {\it spread set} associated to a presemifield $\cS$ is the set of endomorphisms of $\Fq^n$ defined by left-multiplication, namely
\[
\C_\cS := \{L_x:y\mapsto x\circ y\}.
\]
\end{definition}

It is not difficult to see that $\C_\cS$ can be identified with $\C_1(T_\cS)$. Furthermore every nonzero element of $\C_{\cS}$ is invertible. Thus presemifields can equally be studied as $n$-dimensional subspaces of $M_n(\Fq)$. The following result can be found in \cite{Lavrauw2013}, and is a special case of Lemma \ref{lem:span} above.

\begin{lemma}
The tensor rank of a presemifield $\cS$ of dimension $n$ over $\Fq$ is equal to the minimum number of rank one matrices in $M_n(\Fq)$ whose span contains a spread set associated to $\cS$.
\end{lemma}

As mentioned above, isotopic presemifields have the same tensor rank. We note here that there is a natural $S_3$-action on tensors in $\FF^n\otimes \FF^n\otimes \FF^n$ which preserves nonsingularity, which produces  a set of up to six isotopism classes of presemifields associated to a single presemifield, known as the {\it Knuth orbit} \cite{Knuth1965}. Again we refer to \cite{Lavrauw2013} for the following.
\begin{lemma}
The tensor rank of a presemifield is a Knuth orbit invariant.
\end{lemma}

Thus when calculating the tensor rank of a presemifield, we may choose any representative for any of the isotopism classes of its Knuth orbit.


We end this section, with the following lemma which will assist with speeding up the calculation of the tensor rank of $\Fqn$ over $\Fq$.

\begin{lemma}
The automorphism group of a spread set of the field $\Fqn$ in $M_n(\Fq)$ acts transitively on rank one matrices.
\end{lemma}

\begin{proof}
It is well-known that the ring of $n\times n$ matrix over $\Fq$ is isomorphic to the ring of linearized polynomials of degree at most $q^{n-1}$ (that is, polynomials in $\Fqn[x]$ of the form $f(x) = \sum_{i=0}^{n-1}f_i x^{q^i}$) with operations polynomial addition and composition modulo $x^{q^n}-x$. The rank of a linearized polynomial is defined as its rank as an $\Fq$-linear map from $\Fqn$ to itself.

The image of a spread set of the field $\Fqn$ corresponds to the set $\{ax:a\in \Fqn\}$, which by abuse of notation we will refer to as $\C_{\Fqn}$. It is straightforward to verify that the automorphism group of $\C_{\Fqn}$ contains the set $\{(ax,bx):a,b\in \Fqn^\times\}$. In fact the full automorphism group is the set $\{(ax^{q^i},bx^{q^{n-i}}):a,b\in \Fqn^\times\}$, but we do not require this fact.

It is also well-known that the set of rank one matrices corresponds to the set $\{a\Tr(bx):a,b\in \Fqn^\times\}$, where $\Tr(x)= \sum_{i=0}^{n-1}x^{q^i}$ is the linearized polynomial corresponding to the relative trace map from $\Fqn$ to $\Fq$. It is straightforward to check that each of these elements has rank one, since the image of $a\Tr(bx)$ is the one-dimensional $\Fq$-subspace defined by $a$. A simple counting argument verifies that all rank one elements are of this form.

Now consider the element $\Tr(x)$ of rank one. Its image under the element $(ax,bx)$ is $(ax)\circ \Tr(x) \circ (bx) = a\Tr(bx)$ for any $a,b\in \Fqn^\times$. Thus the automorphism group of a spread set of $\Fqn$ acts transitively on rank one matrices, as claimed.
\end{proof}

%

\subsection{Tensors and Linear Codes}\label{subsec:tensors_and_codes}

In this section we generalise Brockitt and Dobkin's result from \cite{BrDo1978}, given for concise 3-fold tensors over $\F_2$, to general $t$-fold tensors over $\Fq$. The generalisation from concise $3$-fold tensors over $\F_2$ to concise $3$-fold tensors over $\Fq$ is straightforward and can be found in Burgisser et al. \cite[Theorem 18.4]{BuClSh1997} where the result is rightfully attributed to Brockitt and Dobkin.

The result relates the tensor rank of a tensor $T$ to the minimum distance of certain linear block codes associated to $T$.

Consider a decomposition of a tensor $T\in \bigotimes_{i=1}^t V_i$ into the sum of $R$ pure tensors, say 
$$T = \sum_{j=1}^R v_{j1}\otimes \cdots \otimes v_{jt}.$$ 
For $i\in \{1,\ldots,t\}$, define 
$$\G_i = 
\left (
\begin{array}{ccccccc}
v_{1i}^T & | & v_{2i}^T & | & \cdots & | &v_{Ri}^T
\end{array}
\right ),
$$
regarded as an ($\dim(V_i)\times R$)-matrix over $\Fq$. Denote the linear code generated by the rows of $\G_i$ by $C_i$. Then $C_i$ is a linear code (in the Hamming metric) of length $R$ and dimension at most $\dim(V_i)$. The codewords of $C_i$ are those of the form $(v^\vee(v_{1i}),v^\vee(v_{2i}),\ldots,v^\vee(v_{Ri}))$ for some $v^\vee\in V_i^\vee$. Note that decompositions of a tensor are far from unique, and so these codes are not canonically defined by a tensor; rather, they are defined by a decomposition of a tensor. 

Note that $\dim(C_i)=\dim(V_i)$ if and only if $T$ is $i$-concise; for the matrix $\G_i$ does not have full row rank if and only if there exists $v^\vee\in V_i^\vee$ with $v^\vee(v_{ji})=0$ for all $j\in \{1,\ldots,R\}$, if and only if there exists $v^\vee\in V_i^\vee$ such that $v^\vee(T)=0$, if and only if $T$ is not $i$-concise.



\begin{lemma}\label{lem:mindist}
For each $i\in \{1,\ldots,t\}$, the minimum distance of the code $C_i$ is at least the minimum tensor rank of a nonzero element of the $i$-th contraction space $\C_i(T)$.
\end{lemma}

\begin{proof}
Consider an element $v^\vee\in V_i^\vee$, the corresponding contraction $v^\vee(T)$, and the corresponding codeword 
$$c=(v^\vee(v_{1i}),v^\vee(v_{2i}),\ldots,v^\vee(v_{Ri})).$$
Let $J\subset \{1,\ldots,R\}$ denote the support of $c$; that is, $J=\{j:v^\vee(v_{ji})\ne 0\}$. Then 
$$v^\vee(T) = \sum_{j\in J} v^\vee(v_{ji})  (v_{j1}\otimes \cdots \otimes v_{j,i-1}\otimes v_{j,i+1}\otimes  \cdots \otimes v_{jt}),
$$
and so the tensor rank of $v^\vee(T)$ is at most $|J|=\mathrm{wt}(c)$, implying the result.
\end{proof}

A generally weaker but more easy to apply result is the following. We note that these two lemmas coincide in the case $t=3$.

\begin{lemma}
For each $i\in \{1,\ldots,t\}$, the minimum distance of the code $C_i$ is at least 
$$\min\{\max\{\dim(v^\vee(\C_j(T))):j \ne i\}:v^\vee\in V_i^\vee\}.$$
\end{lemma}

\begin{proof}
Let $v^\vee\in V_i^\vee$ and let $v_j^\vee\in V_j^\vee$, and let $c$ be the codeword of $C_i$ corresponding to $v^\vee$. As in the proof of Lemma \ref{lem:mindist}, $v^\vee(T)$ has tensor rank at most $\mathrm{wt}(c)$, and so its contraction space $\C_j(v^\vee(T))$ has dimension at most $\mathrm{wt}(c)$. Then it is straightforward to verify that $\C_j(v^\vee(T))= v^\vee(\C_j(T))$ by definition, proving the result.
\end{proof}

The following theorem is a generalisation of a theorem, for concise 3-fold tensors, of Brockett and Dobkin \cite{BrDo1978}, and Burgisser et al. \cite[Theorem 18.4]{BuClSh1997}. The theorem relates the tensor rank to the minimal length of a linear code with given dimension and minimum distance. We let $N_q(k,d)$ denote the shortest $\Fq$-linear code with dimension $k$ and minimum Hamming distance $d$.

\begin{theorem}\label{lem:genbound}
The tensor rank of a tensor $T$ is at least $N_q(\dim(C_i),d_i)$, for each $i\in \{1,\ldots,t\}$, where $d_i$ is the minimum tensor rank of an element of $\C_i(T)$.
\end{theorem}

\begin{proof}
This follows immediately from Lemma \ref{lem:mindist}, for if $T$ has tensor rank $R$, then $C_i$ is an $[R,\dim(C_i),\geq d_i]$ code, and so $R\geq N_q(\dim(C_i),d_i)$ by definition.
\end{proof}

In the case of the tensor of a presemifield, we have that the dimension of $C_i$ is $n$, and every nonzero element of each contraction space is an invertible matrix, and hence $d_i\geq n$. Thus we obtain the following.
\begin{lemma}\label{lem:bound}
The tensor rank of a presemifield of dimension $n$ over $\Fq$ is at least $N_q(n,n)$.
\end{lemma}



The following results are geared towards assisting with the algorithmic calculation of the tensor rank. Consider the following two sets associated with a minimal decomposition of $T$.
Let 
$$S=\{v_{j1}\otimes \cdots \otimes v_{jt}:j\in \{1,\ldots,R\}\},$$ where $T$ has rank $R$ and is the sum of the elements of $S$, and 
$$S_i = \{v_{j1}\otimes \cdots \otimes v_{j,i-1}\otimes v_{j,i+1}\otimes  \cdots \otimes v_{jt}:j\in \{1,\ldots,R\}\}.$$

%
\begin{lemma}\label{lem:sub2gen}
Let $d$ be such that there does not exist an $\Fq$-linear $[R,\dim(C_i),d+1]$-code. Then the set $S_i$ contains a subset $U$ of size $d-1$ such that the span $\langle U,\C_i(T)\rangle$ contains at least one further element of $S_i$.
\end{lemma}

\begin{proof}
Let $c$ be a codeword of minimum weight $m$ in $C_i$, determined by $v^\vee\in V_i^\vee$. Then $m\leq d$ and $S_i$ contains a subset $W$ consisting of $m$ tensors whose span contains $v^\vee(T)$. Let $W'=W\backslash \{w\}$ for some $w\in W$, and let $U$ be the union of $W'$ and any $d-m$ elements of $S_i\backslash W$. Then $w\in \langle U,\C_i(T)\rangle$, proving the claim.
\end{proof}

\begin{lemma}\label{lem:sub2}
Let $\cS$ be a semifield of order $q^n$ with a spread set $\C$, and suppose the tensor rank of $\cS$ is $R$. If there does not exist an $[R,n,n+1]$ code over $\Fq$, then there exists a $(2n-1)$-dimensional subspace containing $\C$ containing at least $n$ linearly independent elements of rank one.
\end{lemma}

\begin{proof}
This follows from Lemma \ref{lem:sub2gen} taking $d=n$, given that $\dim(C_i)=n$. Using the notation of that lemma, the space $\langle U\rangle$ is disjoint from $\C$ as elements of $\langle U \rangle$ have rank at most $n-1$, while nonzero elements of $\C$ have rank $n$. Furthermore as the $n-1$ elements of $U$ and one further element $w$ of rank one span an element of $\C$ of rank $n$, the elements of $U$ must be linearly independent. This shows that $\langle U,\C_i(T)\rangle$ has dimension $2n-1$.
%
\end{proof}

\begin{lemma}\label{lem:sub}
If the tensor rank of a semifield of dimension $n$ over $\Fq$ is $R$, then for any $k\leq n$, its spread set contains a $(n-k)$-dimensional subspace of tensor rank at most $R-k$.
\end{lemma}

\begin{proof}
Suppose $\C$ is a semifield spread set, contained in the span of the $R$ rank one matrices $A_1,\ldots,A_R$. Then the subspace $\C\cap \langle A_1,\ldots,A_{R-k}\rangle$ has dimension $\geq n-k$, and has tensor rank at most $R-k$ by construction.
\end{proof}

\subsection{Known Tensor Ranks}

The following is well known.
\begin{lemma}
The tensor rank of the field $\Fqn$ over $\Fq$ is at least $2n-1$, with equality if and only if $q\geq 2n-1$.
\end{lemma}

One direction of this result follows from Lemma \ref{lem:bound} and the well-known results for $N_q(n,n)$, while the other direction follows from a direct construction of a tensor representing $\Fqn$, obtained via polynomial interpolation \cite{BuClSh1997}. It is known that the tensor rank of $\Fqn$ over $\Fq$ is at most $C_q n$, where $C_q$ is a constant. Further bounds and asymptotic results for the tensor rank of $\Fqn$ are known; we refer to for example \cite{Ballet} for further results in this direction. We detail one such result from \cite{Ballet} for context. 

\begin{lemma}
The tensor rank of a field of dimension $n$ over $\FF_2$ is at most 
$$\min\{38n,\frac{477}{26}n + \frac{45}{2}\}.$$
\end{lemma}

In \cite{LaPaZa2013} the tensor ranks of some small presemifields were calculated precisely.
\begin{lemma}
The tensor rank of a semifield of order $2^3$ or $3^3$ is precisely $6$.
\end{lemma}
This last two lemmas give the following.
\begin{lemma}
The tensor rank of a field of degree $3$ over $\Fq$ is precisely $6$ if $q\leq 3$, and precisely $5$ if $q>3$.
\end{lemma}
This leaves the smallest open cases to be those of presemifields of order $2^4,2^5$, and $3^4$. In this paper we complete the first and third cases.

\section{Algorithms}\label{sec:algorithms}

There are two main algorithms used in this work, denoted Algorithms 2 and 3 below. Algorithm 2 builds up subspaces of $M_n(\Fq)$ spanned by matrices of rank one, up to equivalence under the action of $G\cong \GL(\Fq^n)\times \GL(\Fq^n)$, and looks for semifield spread sets contained therein. Algorithm 3 begins with a semifield spread set, and adds in rank one matrices until the resulting space is spanned by the rank one elements contained in it, using the automorphism group of the semifield spread set at each step. Which of these algorithms is more appropriate for a particular case depends on the tensor rank (which we do not know a priori), and the size of the automorphism group of the spread set (which we do know). Algorithm 2 is the best option for low tensor rank and small automorphism group, whereas Algorithm 3 is the best option for high tensor rank and large automorphism group. 

The two authors of this paper implemented the algorithms independently in two different systems (the FinInG package for GAP, and MAGMA, respectively). In GAP, equivalence was tested using group theoretic and finite geometry techniques. In MAGMA, equivalence was tested using linear algebraic techniques, using modifications of the algorithms presented in \cite{Rua2009}, \cite{Rua2011a}. We will refer to both of these algorithms as {\bf EquivalenceClasses}, which takes as input a list of semifield spread sets (and an optional subgroup of $\GL(\Fq^n)\times \GL(\Fq^n)$), and returns a list of representatives for the distinct equivalence classes of the input.

We now give a conceptual description of these algorithms. The following notation is used. Let $e_1,\ldots,e_n$ be the standard basis for the vector space of $n\times 1$ matrices over $\Fq$. Define $\diag=\langle e_ie_i^T~:~i\in \{1,\ldots,n\}\rangle$ as the space of diagonal matrices. Let $X$ denote the set of rank one matrices. If $a$ denotes a subspace of $M_n(\Fq)$ then $a^*$ denotes the set of nonzero elements of $a$.

\subsection{Finding semifield spread sets in a given space}
The first algorithm takes as input a subspace $S\leq M_n(\Fq)$ and finds all (semifield) spreadsets contained in $S$. 
\begin{itemize} 
\item[] Algorithm 1: {\bf SpreadSets}
\item[] {\bf input} a subspace $S\leq M_n(\Fq)$ 
\item[(Step 1)]  $E_i = \{a\in S:\det(a)\ne 0, a_1=e_i\}$, where $a_1$ denotes the first row of $a$.
\item[(Step 2)]  $A=\{\langle a\rangle:a\in E_1\}$, $m=1$
\item[(Step 3)]  $A=\{\langle a,b\rangle:a\in A,b\in E_{m+1}\}$, $m=m+1$
\item[(Step 4)]  $A=\{a \in A:\forall x\in a^* \det(x)\ne 0\}$
\item[(Step 5)]  if $m=n$ {\bf return} SpreadSets($S$) := EquivalenceClasses($A$)
\begin{itemize}
\item[] Otherwise repeat Steps 3-5.
\end{itemize}
\end{itemize}

\subsection{Building spread sets of given tensor ranks}\label{sec:byrankone}

At each step we extend each subspace of $M_n(\Fq)$ with an element of rank one, calculate representatives for the equivalence classes, and calculate the spread sets contained in each representative.

\begin{itemize}
\item[] Algorithm 2: {\bf SpreadSetsByRank}
\item[] {\bf input} an integer $R\geq n$ 
\item[(Step 1)] $A=\{\diag\}$, $m=n$.
\item[(Step 2)] $A = \{\langle a,b\rangle:a\in A,b\in X,b\notin a\}$, $m=m+1$.
\item[(Step 3)] $A$ = EquivalenceClasses$(A)$
\item[(Step 4)] $SS_m = \cup_{a\in A}$SpreadSets($a$)
\item[(Step 5)] if $m=R$ {\bf return} SpreadSetsByRank(R) := EquivalenceClasses($SS_m$)
\begin{itemize}
\item[] Otherwise repeat Steps 2-5.
\end{itemize}
\end{itemize}

The output contains representatives for all isotopism classes of semifields of tensor rank at most $R$.

\subsection{Calculating the tensor rank of a given semifield spread set}

This algorithm can be used to verify the tensor rank of a given spread set $\C$.


\begin{itemize}
\item[] Algorithm 3: {\bf TensorRank}
\item[] {\bf input} a spread set $\C\leq M_n(\Fq)$ with automorphism group $G_\C$.
\item[(Step 1)] $A=\{\C\}$, $m=n$
\item[(Step 2)]  $A = \{\langle a,b\rangle:a\in A,b\in X,b\notin a\}$, $m=m+1$.
\item[(Step 3)]  If $\exists a\in A$ such that $\langle x \in a:\rank(x)=1 \rangle=a$ then {\bf return TensorRank($\C$) := $m$}.
\begin{itemize}
\item[] Otherwise $A$ = EquivalenceClasses($A,G_\C$); repeat Step 2.
\end{itemize}
\end{itemize}

{\bf Notation.} We will sometimes use a more concise representation for a matrices in $M_n(\FF_p)$ and write
$$
\sum_{i,j} a_{ij}p^{(i-1)n+(j-1)}
$$
for a matrix with $(i,j)$-entry $a_{ij}$.

\section{Semifields of order $2^4$}\label{sec:16}


There are 3 isotopism classes of semifields of order $2^4$, falling into $3$ Knuth orbits. These were first classified by Knuth \cite{Knuth1965}. We denote them by $\FF_{2^4}$, $\mathcal{S}_1$, and $\mathcal{S}_2$, where $\mathcal{S}_1$ is the unique semifield of order $2^4$ containing a nucleus of order $2^2$.


The following result was noted in \cite{ByNeRaSh}, based on computer classifications by the second author of this paper. It was calculated using the above algorithms, and we present proof below. It was further noted in \cite{ByNeRaSh} that a spread set of the field $\FF_{2^4}$ contains no hyperplane (three-dimensional subspace) with tensor rank $7$, whereas the spread sets of the other two classes of semifields do possess hyperplanes with tensor rank $7$.

\begin{theorem}\label{thm:16}
The tensor rank of every semifield of order $2^4$ is nine.
\end{theorem}

\begin{proof}
By Lemma \ref{lem:bound}, any semifield of order $2^4$ has tensor rank at least eight, as  $N_2(4,4)=8$ \cite{codetables}.

We utilise the algorithm in Section \ref{sec:byrankone}. We retain only the six-dimensional spaces containing a partial spread set of dimension two, and only the seven-dimensional spaces containing a partial spread set of dimension three. This ensures that we obtain a representative for every semifield of tensor rank eight, should one exist.

We find 19 equivalence classes of five-dimensional spaces spanned by rank one matrices. To representatives of these classes we add each matrix of rank one and classify them up to equivalence, finding 236 equivalence classes of six-dimensional spaces. Of these, 33 contained a partial spread set of dimension 2; we retain these 33 representatives and discard the rest. To each of these representatives we add each rank one matrix, finding 4371 seven-dimensional spaces forming 910 equivalence classes. Of these only two contain a partial spread set of dimension 3; again we discard the rest. To each of these two representatives we add each rank one matrix, finding 201 eight-dimensional spaces forming 23 equivalence classes. None of these contain a spread set, proving that the tensor rank of any semifield of order $2^4$ is at least nine.


Thus it suffices to show that a spread set for each is contained in a space spanned by nine matrices of rank one.

The following are representations of bases of a spread set of each semifield. \[
\begin{array}{lrrrr}
\FF_{2^4}&33825& 14402& 25476& 50744\\
\mathcal{S}_1&33825& 51250& 63940& 24136\\
\mathcal{S}_2&33825& 22594& 11684& 51864
\end{array}
\]
A computer calculation implementing Algorithm 2 showed that each of these spread sets is indeed contained in a space spanned by nine matrices of rank one. The nine rank one matrices spanning each spread set are as follows.
\[
\begin{array}{lrrrrrrrrr}
\FF_{2^4}&85& 8738& 57582& 32896& 1632& 3072& 30576& 53261& 4096\\ 
\mathcal{S}_1&4112& 238& 80& 53469& 4353& 2304& 13059& 2570& 26112 \\
\mathcal{S}_2&1& 204& 53456& 39321& 4080& 2048& 43530& 47872& 57344
\end{array}
\]
Thus each semifield of order $2^4$ has rank nine, as claimed.
\end{proof}

Expanded details of these facts can be found in the next sections. We note that these sets of rank one matrices are not unique, even up to equivalence. For example, there are five equivalence classes of nine-dimensional spaces spanned by rank one matrices containing the spread set of $\FF_{16}$. The chosen decompositions are arbitrary.

\subsection{The field $\FF_{2^4}$}

A basis for a spread set for the field $\FF_{2^4}$ is
\[
    \npmatrix{ 1& 0& 0& 0\\ 0& 1& 0& 0\\ 0& 0& 1& 0\\ 0& 0& 0& 1 },
    \npmatrix{ 0& 1& 0& 0\\ 0& 0& 1& 0\\ 0& 0& 0& 1\\ 1& 1& 0& 0 },
    \npmatrix{ 0& 0& 1& 0\\ 0& 0& 0& 1\\ 1& 1& 0& 0\\ 0& 1& 1& 0 },
    \npmatrix{ 0& 0& 0& 1\\ 1& 1& 0& 0\\ 0& 1& 1& 0\\ 0& 0& 1& 1 }
\]
This is contained in the space spanned by the following nine rank one matrices.
\[
    \npmatrix{ 1& 0& 1& 0\\ 1& 0& 1& 0\\ 0& 0& 0& 0\\ 0& 0& 0& 0 },
    \npmatrix{ 0& 1& 0& 0\\ 0& 1& 0& 0\\ 0& 1& 0& 0\\ 0& 1& 0& 0 },
    \npmatrix{ 0& 1& 1& 1\\ 0& 1& 1& 1\\ 0& 0& 0& 0\\ 0& 1& 1& 1 },
    \npmatrix{ 0& 0& 0& 0\\ 0& 0& 0& 1\\ 0& 0& 0& 0\\ 0& 0& 0& 1 },
    \npmatrix{ 0& 0& 0& 0\\ 0& 1& 1& 0\\ 0& 1& 1& 0\\ 0& 0& 0& 0 },
    \]
    \[
    \npmatrix{ 0& 0& 0& 0\\ 0& 0& 0& 0\\ 0& 0& 1& 1\\ 0& 0& 0& 0 },
    \npmatrix{ 0& 0& 0& 0\\ 1& 1& 1& 0\\ 1& 1& 1& 0\\ 1& 1& 1& 0 },
    \npmatrix{ 1& 0& 1& 1\\ 0& 0& 0& 0\\ 0& 0& 0& 0\\ 1& 0& 1& 1 },
    \npmatrix{ 0& 0& 0& 0\\ 0& 0& 0& 0\\ 0& 0& 0& 0\\ 1& 0& 0& 0 }
\]
%
%


\subsection{The semifield $\mathcal S_1$ of order $2^4$ with a nucleus of order $2^2$}

The semifield with basis
\[
    \npmatrix{ 1& 0& 0& 0\\ 0& 1& 0& 0\\ 0& 0& 1& 0\\ 0& 0& 0& 1 },\quad
    \npmatrix{ 0& 1& 0& 0\\ 1& 1& 0& 0\\ 0& 0& 0& 1\\ 0& 0& 1& 1 },\quad
    \npmatrix{ 0& 0& 1& 0\\ 0& 0& 1& 1\\ 1& 0& 0& 1\\ 1& 1& 1& 1 },\quad
    \npmatrix{ 0& 0& 0& 1\\ 0& 0& 1& 0\\ 0& 1& 1& 1\\ 1& 0& 1& 0 }
\]
is spanned by the following nine rank one matrices:
\[
    \npmatrix{ 0& 0& 0& 0\\ 1& 0& 0& 0\\ 0& 0& 0& 0\\ 1& 0& 0& 0 },\quad
    \npmatrix{ 0& 1& 1& 1\\ 0& 1& 1& 1\\ 0& 0& 0& 0\\ 0& 0& 0& 0 },\quad
    \npmatrix{ 0& 0& 0& 0\\ 1& 0& 1& 0\\ 0& 0& 0& 0\\ 0& 0& 0& 0 },\quad
    \npmatrix{ 1& 0& 1& 1\\ 1& 0& 1& 1\\ 0& 0& 0& 0\\ 1& 0& 1& 1 },\quad
    \npmatrix{ 1& 0& 0& 0\\ 0& 0& 0& 0\\ 1& 0& 0& 0\\ 1& 0& 0& 0 },
    \]
    \[
    \npmatrix{ 0& 0& 0& 0\\ 0& 0& 0& 0\\ 1& 0& 0& 1\\ 0& 0& 0& 0 },\quad
    \npmatrix{ 1& 1& 0& 0\\ 0& 0& 0& 0\\ 1& 1& 0& 0\\ 1& 1& 0& 0 },\quad
    \npmatrix{ 0& 1& 0& 1\\ 0& 0& 0& 0\\ 0& 1& 0& 1\\ 0& 0& 0& 0 },\quad
    \npmatrix{ 0& 0& 0& 0\\ 0& 0& 0& 0\\ 0& 1& 1& 0\\ 0& 1& 1& 0 }
\]

\subsection{The semifield $\mathcal S_2$ of order $2^4$ with trivial nuclei}

The semifield with basis
\[
    \npmatrix{ 1& 0& 0& 0\\ 0& 1& 0& 0\\ 0& 0& 1& 0\\ 0& 0& 0& 1 },\quad
    \npmatrix{ 0& 1& 0& 0\\ 0& 0& 1& 0\\ 0& 0& 0& 1\\ 1& 0& 1& 0 },\quad
    \npmatrix{ 0& 0& 1& 0\\ 0& 1& 0& 1\\ 1& 0& 1& 1\\ 0& 1& 0& 0 },\quad
    \npmatrix{ 0& 0& 0& 1\\ 1& 0& 0& 1\\ 0& 1& 0& 1\\ 0& 0& 1& 1 }
\]

is spanned by the following nine rank one matrices:
\[
    \npmatrix{ 1& 0& 0& 0\\ 0& 0& 0& 0\\ 0& 0& 0& 0\\ 0& 0& 0& 0 },\quad
    \npmatrix{ 0& 0& 1& 1\\ 0& 0& 1& 1\\ 0& 0& 0& 0\\ 0& 0& 0& 0 },\quad
    \npmatrix{ 0& 0& 0& 0\\ 1& 0& 1& 1\\ 0& 0& 0& 0\\ 1& 0& 1& 1 },\quad
    \npmatrix{ 1& 0& 0& 1\\ 1& 0& 0& 1\\ 1& 0& 0& 1\\ 1& 0& 0& 1 },\quad
    \npmatrix{ 0& 0& 0& 0\\ 1& 1& 1& 1\\ 1& 1& 1& 1\\ 0& 0& 0& 0 },
    \]
    \[
    \npmatrix{ 0& 0& 0& 0\\ 0& 0& 0& 0\\ 0& 0& 0& 1\\ 0& 0& 0& 0 },\quad
    \npmatrix{ 0& 1& 0& 1\\ 0& 0& 0& 0\\ 0& 1& 0& 1\\ 0& 1& 0& 1 },\quad
    \npmatrix{ 0& 0& 0& 0\\ 0& 0& 0& 0\\ 1& 1& 0& 1\\ 1& 1& 0& 1 },\quad
    \npmatrix{ 0& 0& 0& 0\\ 0& 0& 0& 0\\ 0& 0& 0& 0\\ 0& 1& 1& 1 }
\]


\section{Semifields of order $3^4$}\label{sec:81}

There are 27 isotopism classes of semifields of order $3^4$, falling into $12$ Knuth orbits. These were first classified by Dempwolff \cite{Demp2008}, who labelled the Knuth orbits as I-XII. 

By Lemma \ref{lem:bound}, any semifield has tensor rank at least eight, as $N_3(4,4)=8$ \cite{codetables}. As we will see in the forthcoming sections, the tensor rank of these semifields is either 8 or 9. 
This is the first result establishing the fact that two semifields of the same order can have different tensor rank. This shows that the tensor rank of a semifield is a non-trivial invariant for finite semifields.
It also shows that from a computational point of view, for a fixed order, certain semifields will perform better than others. In fact, it turns out that the finite field and the GTF of order $3^4$ are the worst choices in terms of complexity of multiplication, since they both have tensor rank 9, while the other semifields of order $3^4$ have tensor rank 8.

\subsection{The field $\FF_{3^4}$}

This is labelled as XII in \cite{Demp2008}. We choose a different basis for convenience. A basis for a spread set of $\FF_{3^4}$ is the following:
\[
\npmatrix{ 1& 0& 0& 0\\ 0& 1& 0& 0\\ 0& 0& 1& 0\\ 0& 0& 0& 1 },\quad
\npmatrix{ 0& 1& 0& 0\\ 0& 0& 1& 0\\ 0& 0& 0& 1\\ 1& 0& 0& 1 },\quad
\npmatrix{ 0& 0& 1& 0\\ 0& 0& 0& 1\\ 1& 0& 0& 1\\ 1& 1& 0& 1 },\quad
\npmatrix{ 0& 0& 0& 1\\ 1& 0& 0& 1\\ 1& 1& 0& 1\\ 1& 1& 1& 1 }
\]

\begin{lemma}\label{lem:F81lower}
The tensor rank of $\FF_{3^4}$ is at least nine.
\end{lemma}
\begin{proof}
Let $\C$ be a spread set for $\FF_{3^4}$. By Lemma \ref{lem:sub2}, if $\FF_{3^4}$ has tensor rank eight, then since there does not exist an $[8,4,5]$ code over $\FF_3$, there must exist three elements $A_1,A_2,A_3$ of $M_4(\FF_3)$ of rank one such that $\langle \C,A_1,A_2,A_3\rangle$ contains an element of rank one which linearly independent from $A_1,A_2,A_3$.

Since the automorphism group of $\C$ acts transitively on matrices of rank one, we may choose $A_1$. We calculate the orbits of rank one matrices under the automorphism group of $\langle \C,A_1\rangle$ and form representatives for the six-dimensional spaces spanned by $\C$ and two matrices of rank one; we find that there are 662 of these in total. We then form all seven-dimensional spaces spanned by one of these representatives and another matrix of rank one; we find 763858 spaces.

By the above, we may discard any of these spaces which do not contain four linearly independent matrices of rank one. We are left with 5078 spaces. For each of these we form all eight-dimensional spaces containing one of these spaces and another matrix of rank one. We find that none of these spaces contain eight linearly independent matrices of rank one, proving that the tensor rank of $\FF_{3^4}$ is at least nine. 
\end{proof}


\begin{theorem}\label{lem:at_most_nine}
The tensor rank of $\FF_{3^4}$ is exactly nine.
\end{theorem}
\begin{proof}
By Lemma \ref{lem:F81lower}, the tensor rank of $\FF_{3^4}$ is at least nine. Thus it suffices to show that the spread set of $\FF_{3^4}$ is contained in a space spanned by nine matrices of rank one.

One verifies that the spread set for $\F_{3^4}$ is contained in the space spanned by the following nine matrices of rank one.
\[
\npmatrix{ 1& 0& 0& 2\\ 2& 0& 0& 1\\ 1& 0& 0& 2\\ 0& 0& 0& 0 },\quad
\npmatrix{ 1& 2& 2& 1\\ 2& 1& 1& 2\\ 0& 0& 0& 0\\ 0& 0& 0& 0 },\quad
\npmatrix{ 0& 0& 0& 0\\ 0& 0& 1& 1\\ 0& 0& 1& 1\\ 0& 0& 2& 2 },\quad
\npmatrix{ 1& 0& 0& 0\\ 1& 0& 0& 0\\ 1& 0& 0& 0\\ 1& 0& 0& 0 },\quad
\npmatrix{ 1& 2& 1& 0\\ 0& 0& 0& 0\\ 1& 2& 1& 0\\ 2& 1& 2& 0 },
\]
\[
\npmatrix{ 0& 0& 0& 0\\ 1& 1& 1& 1\\ 2& 2& 2& 2\\ 1& 1& 1& 1 },\quad
\npmatrix{ 0& 0& 1& 0\\ 0& 0& 1& 0\\ 0& 0& 1& 0\\ 0& 0& 0& 0 },\quad
\npmatrix{ 0& 0& 0& 0\\ 0& 0& 0& 0\\ 0& 1& 0& 0\\ 0& 1& 0& 0 },\quad
\npmatrix{ 0& 1& 0& 2\\ 0& 1& 0& 2\\ 0& 2& 0& 1\\ 0& 0& 0& 0 }
\]
\end{proof}


As seen in Section \ref{subsec:tensors_and_codes}, a minimal decomposition of a tensor $T$ of tensor rank $R$ gives three linear codes.
The three $n\times R$ matrices $\G_1,\G_2,\G_3$ whose rows generate the linear codes $C_1, C_2, C_3$, obtained from the decomposition implied by the matrices in the proof of Lemma \ref{lem:at_most_nine} are as follows.


 \[
\G_1=\npmatrix{
  1& 1& 0& 1& 1& 0& 1& 0& 1\\
  2& 2& 1& 1& 0& 1& 1& 0& 1\\
  1& 0& 1& 1& 1& 2& 1& 1& 2\\
  0& 0& 2& 1& 2& 1& 0& 1& 0}
  \]
\[
\G_2=\npmatrix{
 1& 1& 0& 1& 1& 1& 0& 0& 0\\
 0& 2& 0& 0& 2& 1& 0& 1& 1\\
 0& 2& 1& 0& 1& 1& 1& 0& 0\\
 2& 1& 1& 0& 0& 1& 0& 0& 2}
 \]
 \[
\G_3=\npmatrix{
 2& 2& 1& 2& 1& 2& 1& 0& 0\\
 1& 1& 1& 0& 1& 2& 0& 0& 0\\
 2& 0& 1& 0& 1& 2& 0& 1& 1\\
 2& 1& 0& 0& 0& 1& 1& 0& 1}
 \]
We note that each of these is a generator matrix for a $[9,4,4]_3$ code. It can be verified that the first two are equivalent to each other. However, the third is inequivalent to the first two. This can easily be verified by computing the weight distribution of the codes which give $[ 1, 0, 0, 0, 6, 24, 24, 12, 12, 2 ]$ for $\mathcal G_1$ and $[ 1, 0, 0, 0, 10, 22, 22, 8, 14, 4 ]$ for $\mathcal G_3$.


Finally, we remark that the minimal decomposition given in the proof of Lemma \ref{lem:at_most_nine} is far from unique. This implies that there are many other triples of linear codes associated to a minimal decomposition of the tensor of $\FF_{3^4}$.


\subsection{The Generalised Twisted Field of order $3^4$}

This is labelled as IX in \cite{Demp2008}. We choose a different basis for convenience, and denote it as $\mathrm{GTF}_{3^4}$. A basis for a spread set of a generalised twisted field of order $3^4$ is the following.
\[
    \npmatrix{ 1& 0& 0& 0\\ 0& 1& 0& 0\\ 0& 0& 1& 0\\ 0& 0& 0& 1 },
    \npmatrix{ 0& 1& 0& 0\\ 1& 1& 0& 0\\ 1& 1& 1& 1\\ 1& 2& 1& 2 },
    \npmatrix{ 0& 0& 1& 0\\ 0& 0& 0& 1\\ 1& 2& 1& 1\\ 2& 0& 1& 2 },
    \npmatrix{ 0& 0& 0& 1\\ 0& 0& 1& 1\\ 2& 0& 2& 1\\ 0& 2& 1& 0 }
\]

\begin{lemma}\label{lem:GTF81lower}
The tensor rank of $\mathrm{GTF}_{3^4}$ is at least nine.
\end{lemma}
\begin{proof}
This result uses the same approach as Lemma \ref{lem:F81lower}, taking into account that the automorphism group of a spread set for $\mathrm{GTF}_{3^4}$ is smaller than that of a spread set for $\FF_{3^4}$. We get ten equivalence classes of five-dimensional spaces, 15425 six-dimensional spaces, and 11236916 seven-dimensional spaces.

Of these, 62649 contain four linearly independent elements of rank one. There are 82422491 eight-dimensional spaces spanned by one of these spaces and an element of rank one; none of which are spanned by eight rank ones. 
\end{proof}

\begin{theorem}
The tensor rank of $\mathrm{GTF}_{3^4}$ is exactly nine.
\end{theorem}
\begin{proof}
By Lemma \ref{lem:GTF81lower}, the tensor rank of $\mathrm{GTF}_{3^4}$ is at least nine. Thus it suffices to verify that the spread set for $\mathrm{GTF}_{3^4}$ is contained in the space spanned by the following nine matrices of rank one.
\[
    \npmatrix{ 0& 0& 0& 0\\ 0& 0& 0& 1\\ 0& 0& 0& 1\\ 0& 0& 0& 1 },
    \npmatrix{ 1& 0& 2& 1\\ 0& 0& 0& 0\\ 2& 0& 1& 2\\ 0& 0& 0& 0 },
    \npmatrix{ 0& 0& 0& 0\\ 1& 1& 2& 2\\ 0& 0& 0& 0\\ 1& 1& 2& 2 },
    \npmatrix{ 1& 2& 1& 1\\ 1& 2& 1& 1\\ 1& 2& 1& 1\\ 1& 2& 1& 1 },
    \npmatrix{ 0& 0& 0& 0\\ 0& 1& 2& 1\\ 0& 0& 0& 0\\ 0& 2& 1& 2 },
    \]
    \[
    \npmatrix{ 1& 0& 1& 2\\ 2& 0& 2& 1\\ 1& 0& 1& 2\\ 1& 0& 1& 2 },
    \npmatrix{ 1& 2& 0& 1\\ 0& 0& 0& 0\\ 0& 0& 0& 0\\ 1& 2& 0& 1 },
    \npmatrix{ 0& 0& 0& 0\\ 0& 0& 0& 0\\ 0& 0& 0& 0\\ 1& 0& 2& 0 },
    \npmatrix{ 1& 1& 2& 2\\ 0& 0& 0& 0\\ 0& 0& 0& 0\\ 0& 0& 0& 0 }
\]
\end{proof}



The more concise representation of the bases of a spread set for the field and generalized twisted field are tabulated here.
\[
\begin{array}{lrrrr}
\FF_{3^4}&14408200& 15058227& 16660575& 21463326\\
\mathrm{GTF}_{3^4}&14408200& 37463637& 34827984& 8282925\\
\end{array}
\]
The more concise representation of the set of nine rank one matrices spanning the above spread sets are tabulated here.
\[
\begin{array}{l|rrr}
\FF_{3^4}&363259& 5560& 38502864\\&
 538084& 12328135 
	& 21785760\\ &59787& 1614006& 221187 \\ 
	\hline
\mathrm{GTF}_{3^4}&14528241& 426511& 40395672\\ &23137612& 36673317
 & 34435999\\ & 18069028& 10097379&7676
\end{array}
\]


\subsection{The remaining semifields of order $3^4$}
The following are bases for representatives for the set of Knuth orbits of the remaining semifields of order $3^4$. The numbering is taken from \cite{Demp2008}, though different bases are chosen for convenience.

\[
\begin{array}{lrrrr}
I&     5217375& 8168391& 10127682& 27851041 \\
 II&    4604203& 15640965& 26024736& 26930970 \\
III  &   14467492& 14958678& 39188133& 42832017 \\
IV    & 19878561& 24409758& 35533648& 42221016 \\
V   &  2025495& 2627829& 14408200& 33856140 \\
VI   &  5157840& 10668294& 16374159& 28816156 \\
VII   &  8817750& 14467492& 20037945& 31590270 \\
VIII   &  15093225& 30319137& 37030935& 37294588  \\
X   &  14408200& 16058439& 29914524& 37686954 \\
XI   &  22027141& 22483740& 29332053& 33106104
\end{array}
\]


\begin{lemma}
All semifields from the families I-VIII and X-XI have tensor rank eight.
\end{lemma}

\begin{proof}
By Lemma \ref{lem:bound}, the tensor rank of each of these semifields is at least eight. 
One verifies that the spread sets are contained in the space spanned by the eight matrices of rank one as listed below, thus proving that the tensor rank is precisely eight.

The algorithm used for Lemmas \ref{lem:F81lower} and \ref{lem:GTF81lower} are not efficient here, as the automorphism groups are much smaller. Instead we utilise the algorithm in Section \ref{sec:byrankone}. We retain only the six-dimensional spaces containing a partial spread set of dimension two, and only the seven-dimensional spaces containing a partial spread set of dimension three. This ensures that we obtain a representative for every semifield of tensor rank eight. 

We found that each spread set is contained in the span of the four elementary diagonal matrices, represented by $1, 243, 59049,14348907$, and four further rank one matrices noted in the below table.
\[
\begin{array}{lrrrrrrrr}
I&       285649& 13819585&  25824144& 42259239 \\
II&        18834768& 20729397& 25402879& 28081132 \\
III&       325507& 7442224&  18834768& 20982117 \\
IV&       9035896&  18981031& 39280132& 41711436 \\
V&       285649& 13819585&  25824144& 42259239 \\
VI&       285649& 13819585&  25824144& 42259239 \\
VII&       325507&  18834768& 34325839& 41964195 \\
VII&       12329415&  18981031& 39280132& 42518560 \\
X&       423775& 13288075&  18834768& 21520120 \\
XI&       9035896&  18981031& 23363440& 39280132 
\end{array}
\]
\end{proof}

\begin{remark}
Note that we did not know a priori which semifields would have tensor rank eight. If a semifield other than the field or twisted field would have been found to have had tensor rank greater than eight, it would have taken significantly longer to calculate its tensor rank, as the automorphism groups are much smaller than those of the field and generalised twisted field.
\end{remark}

%
%

We note that there are subspaces spanned by eight rank one matrices containing many isotopism classes of semifields of order $3^4$. The maximum is ten; that is, there exists an eight dimensional space spanned by rank ones which contains semifield spread sets in 10 of the 27 isotopism classes (and five of the 10 Knuth orbits).

\section{Conclusion and Future Work}
In the first part of the paper we generalised a well known result of Brockitt and Dobkin, which gives a relation between the rank of a concise 3-fold tensor and the minimal length of certain linear codes, to $t$-fold tensors without the condition of conciseness.  
In the second part of the paper we determined the tensor rank of all semifields of order 16 over $\F_2$ and of order 81 over $\F_3$. All semifields of order 16 have tensor rank nine over $\F_2$. Our results for order 81 establish the tensor rank of semifields as a nontrivial invariant of the isotopism classes of semifields (and of the Knuth orbit). This provides the first proof for the existence of semifields of the same order with different tensor rank. All semifields of order $81$ have tensor rank eight over $\F_3$, except for the finite field $\F_{81}$ and the generalized twisted field of order $81$, which both have rank nine over $\F_3$.

As a consequence of the relation between the tensor rank and the multiplicative complexity, the lower tensor rank of some semifields compared to fields of the same order imply that some semifields are more efficient than finite fields in terms of algebra operations. As many applications of finite fields require extension field operations performed as efficiently as possible, it may be beneficial to instead use a semifield of lower tensor rank, should the property of multiplicative associativity not be crucial to the application. Towards this goal, further results on the tensor rank of semifields over $\FF_2$ in particular are required.

Further future work may involve analysis of the different minimal tensor decompositions for a given semifield, in particular with regard to which linear codes can (or cannot) appear. 

%
%
%

\end{document}